\documentclass[11pt,a4paper]{amsart}
\usepackage[british]{babel}
\usepackage{a4wide}
\usepackage{amsmath}
\usepackage{amssymb}
\usepackage{amsfonts}
\usepackage{amsthm}
\usepackage{color}
\usepackage{array}
\usepackage[font=small, width=0.9\textwidth]{caption}
\usepackage{ifthen}

\theoremstyle{plain}
\newtheorem{theorem}{Theorem}

\newtheorem{lemma}[theorem]{Lemma}

\newtheorem{definition}[theorem]{Definition}

\newcommand{\Z}{\mathbb{Z}}

\theoremstyle{definition}

\newcommand{\unitcube}[1]%
{\begin{picture}(12,12)
\linethickness{0.3pt}
\put(0,0){\line(0,1){3}}
\put(0,9){\line(0,1){3}}
\linethickness{1.0pt}
\put(0,4){\line(0,1){4}}
\ifthenelse{#1=0 \OR #1=1}{\put(-2,6){\line(1,0){4}}}{}
\linethickness{0.3pt}
\put(0,0){\line(1,0){3}}
\put(9,0){\line(1,0){3}}
\linethickness{1.0pt}
\put(4,0){\line(1,0){4}}
\ifthenelse{#1=0 \OR #1=2}{\put(6,-2){\line(0,1){4}}}{}
\end{picture}
}

\newcommand{\fourcube}[4]{%
\begin{picture}(24,24)
\put( 0,12){\unitcube{#1}}
\put(12,12){\unitcube{#2}}
\put( 0, 0){\unitcube{#3}}
\put(12, 0){\unitcube{#4}}
\end{picture}
}

\newcommand{\plusline}{
\begin{picture}(4,12)
\linethickness{0.3pt}
\put(0,0){\line(0,1){3}}
\put(0,9){\line(0,1){3}}
\linethickness{1.0pt}
\put(0,4){\line(0,1){4}}
\put(-2,6){\line(1,0){4}}
\end{picture}
}

\newcommand{\minusline}{
\begin{picture}(4,12)
\linethickness{0.3pt}
\put(0,0){\line(0,1){3}}
\put(0,9){\line(0,1){3}}
\linethickness{1.0pt}
\put(0,4){\line(0,1){4}}
\end{picture}
}

\newcommand{\plus}{
\begin{picture}(4,10)(0,-1)
\linethickness{1.0pt}
\put(2,0){\line(0,1){4}}
\put(0,2){\line(1,0){4}}
\end{picture}
}

\newcommand{\minus}{
\begin{picture}(4,10)(0,-1)
\linethickness{1.0pt}
\put(2,0){\line(0,1){4}}
\end{picture}
}

\allowdisplaybreaks[1]
\begin{document}

\title[Substitution for Higher-Dimensional Paperfolding]
  {Substitution Rules for Higher-Dimensional \\ Paperfolding Structures}

\author{Franz G\"{a}hler}
\author{Johan Nilsson}

\address{Fakult\"{a}t f\"{u}r Mathematik, Universit\"{a}t Bielefeld,\newline
\hspace*{\parindent}Postfach 100131, 33501 Bielefeld, Germany}
\email{$\{$gaehler,jnilsson$\}$@math.uni-bielefeld.de}

\begin{abstract}
We present a general scheme how to construct a substitution rule for
generating $d$-dimensional analogues of the paperfolding structures. 
This substitution is proven to be primitive, so that the translation
action on the hull forms a strictly ergodic dynamical system. 
The substitution admits a coincidence in the sense of Dekking, which 
implies that the dynamical system has pure point spectrum. The same 
then holds true also for the diffraction spectrum. The substitution 
also allows us to give estimates on the complexity of the paperfolding
structures, and to determine topological invariants like the \v{C}ech 
cohomology groups of the hull for dimensions $d\le2$.
\end{abstract}

\subjclass[2010]{52C23, 37B50, 05B45}
\keywords{Aperiodic Tilings, Tiling Dynamics}

\maketitle

\section{Introduction}
The paperfolding sequences are a classical and well known example of
aperiodic sequences studied in mathematics, physics and crystallography.  
Roughly speaking, they are obtained by repeatedly folding an arbitrarily 
long strip of paper in the middle, and then unfold it to obtain a 
sequence of creases, which are of two types -- valleys and crests. 
The one sided paperfolding sequences, \texttt{A014577} in \cite{oeis}, 
are now obtained by reading off the creases starting from the left
of the paper, and the two sided version is obtained by reading off from
the centre, outwards in both directions. The structure and properties 
of the paperfolding sequence have been studied in many articles 
\cite{allouche92,allouche2003,allouche95,davis,dekking2012,dekking82,gardner}
(just to mention a few). In particular, Dekking et al.\ \cite{dekking82}
present a substitution rule on a four letter alphabet for the classical 
1-dimensional paperfolding sequence.

Ben-Abraham et al.\ \cite{benabraham} recently gave a generalisation of
the paperfolding sequence to higher dimensions. Using a recursive procedure, 
they construct paperfolding structures in $d$ dimensions for arbitrary $d$.
Unfortunately, this recursive construction, which we review in Section~2, 
is not a substitution. In this paper, we construct a primitive substitution 
which reproduces the higherdimensional paperfolding structures. The
existence of such a substitution has a number of immediate consequences 
on the properties of the paperfolding patterns, and its knowledge
makes available a wealth of machinery to study further properties.

In particular, having a substitution allows us to study the {\em hull} 
of the pattern, and {\em dynamical systems} given by certain group
actions on the hull.  The hull of a pattern is given by the closure of
the translation orbit of the pattern under a local topology. Under
mild conditions, the hull is a compact space, and the translations 
act continuously on it. For primitive substitution structures, the
translation action on the hull is known to be minimal and uniquely
ergodic, which is related to the fact that primitive substitution
patterns are repetitive and have unique patch frequencies. For
background information on primitive substitution patterns and their
dynamical systems, we refer to the recent monograph \cite{BG}.

As a first step in studying paperfolding structures in $d$ dimensions,
we therefore construct a primitive substitution producing them:

\begin{theorem}
  For any $d\ge1$, there exists a primitive substitution $\mu_d$ which
  produces the paperfolding structures introduced by Ben-Abraham et
  al.\ \cite{benabraham}.
\end{theorem}

Using this substitution, we then prove that the $d$-dimensional
paperfolding structures all admit a coincidence in the sense of
Dekking \cite{dekking78}, which immediately implies the following:

\begin{theorem}
  For any $d\ge1$, the $d$-dimensional paperfolding structures have
  pure point diffraction spectrum, and the dynamical system of the
  translation action on the hull has pure point dynamical spectrum.
\end{theorem}

The existence of a primitive substitution also gives immediate bounds
on the complexity the $d$-dimensional paperfolding structures:

\begin{theorem}
  The number of distinct cubic subpatterns of linear size $n$ 
  in a $d$-dimensional paperfolding structure grows at most as
  $\text{const}\cdot n^d$.
\end{theorem}

By means of the Anderson-Putnam \cite{AP} method, a primitive substitution 
also allows to compute topological invariants of the hull, such as \v{C}ech
cohomology groups. We have determined these cohomology groups for
dimensions $1$ and $2$:

\begin{theorem}
  The hull of the classical 1-dimensional paperfolding structures has
  \v{C}ech cohomology groups
  \[ \check{H}^0=\Z,\quad\check{H}^1=\Z[\textstyle{\frac12}]\oplus\Z. \]
  The hull of the 2-dimensional generalised paperfolding structures has
  \v{C}ech cohomology groups
  \[ \check{H}^0=\Z,\quad\check{H}^1=\Z[\textstyle{\frac12}]^2,\quad
     \check{H}^2=\Z[\textstyle{\frac14}]\oplus
                 \Z[\textstyle{\frac12}]^2\oplus\Z^3\oplus\Z_2.
  \]
\end{theorem}

The outline of the paper is as follows. In Section~\ref{sec:recursion},
we review the recursive construction of Ben-Abraham et \mbox{al.}  
\cite{benabraham}, and also fix the notation, which is mostly adopted
from \cite{benabraham}. In Section~\ref{sec:subst} we then present
the general construction of the paperfolding substitution $\mu_d$ in
$d$ dimensions. Before we proceed to analysing the properties of these
substitutions, we illustrate them with pictures in dimensions $1$ and $2$
(Section~\ref{sec:examples}). In Section~\ref{sec:properties} we then
prove the results announced above.

\section{Recursion}
\label{sec:recursion}

Here, we review the general recursion for the $d$-dimensional
paperfolding structures, as outlined by Ben-Abraham et \mbox{al.}
\cite{benabraham}. We also adopt most of their notation.

Each paperfolding structure is obtained by a sequence of 1-dimensional
folds, each in a hyperplane perpendicular to one of the main
coordinate axes. We apply a fold in each of the $d$ directions, before
we fold in the same direction again, so that we can combine $d$
1-dimensional folds into what we call a $d$-fold.  By this we mean a
sequence of $d$ 1-dimensional folds of a quadratic (or cubic,
hypercubic) ``paper'', which are edge-to-edge, from the negative to the 
positive $x_i$-axis, in the order $i=1,2,\ldots, d$. In each fold,
the paper is bent ``upwards'', producing a valley crease in the
previously unfolded paper stack. Let $S_d(n)$ be the $d$-dimensional 
paperfolding structure after $n$ $d$-folds. That is, $S_d(n)$ can be 
seen as a cubic $d$-dimensional paper that has been folded $n$ times
in each direction, and then was unfolded again. We let the origin be 
in the centre of the paper. The resulting crease pattern, with valleys 
and crests, is what we call the paperfolding structure.

The 1-dimensional paperfolding sequence is a sequence defined on the
alphabet $A=\{+,-\}$ ($+$ for valley and $-$ for crest) by the
recursion
\[	S_1(n+1) = m'S_1(n)\ + \ S_1(n),
\]
with the initial condition $S_1(0)=\emptyset$, and where $m'$ means
the reflected sequence, with valleys and crests exchanged. By 
\emph{reflection} we shall from here on mean the reverse of the 
sequence of creases, followed by the swapping their signs.

In the 2-dimensional case the recursion becomes
\begin{center}
\begin{picture}(156,96)(-60,0)
\put(-60,45){{\small$S_2(n+1) = $}}
\put( 50,15){\makebox[48pt][c]{{\small$m_2' S_2(n)$}}}
\put( 50,75){\makebox[48pt][c]{{\small$S_2(n)$}}}
\put(-15,15){\makebox[52pt][c]{{\small$m'_1 m_2' S_2(n)$}}}
\put(-15,75){\makebox[52pt][c]{{\small$m_1' S_2(n)$}}}
\linethickness{0.3pt}\put( 0,48){\line(1,0){3}}
\linethickness{1.0pt}\put( 4,48){\line(1,0){4}}
\linethickness{0.3pt}\put( 9,48){\line(1,0){3}}
\multiput(12,48)(2,0){6}{\put( 0,0){\line(1,0){1}}}
\multiput(24,48)(12,0){4}
{\linethickness{0.3pt}\put( 0,0){\line(1,0){3}}
 \linethickness{1.0pt}\put( 4,0){\line(1,0){4}}
 \linethickness{0.3pt}\put( 9,0){\line(1,0){3}}
}
\multiput(72,48)(2,0){6}{\put( 0,0){\line(1,0){1}}}
\linethickness{1.0pt}
\put(54,46){\line(0,1){4}}
\put(66,46){\line(0,1){4}}
\put(90,46){\line(0,1){4}}
\linethickness{0.3pt}\put(84,48){\line(1,0){3}}
\linethickness{1.0pt}\put(88,48){\line(1,0){4}}
\linethickness{0.3pt}\put(93,48){\line(1,0){3}}
\linethickness{0.3pt}\put(48, 0){\line(0,1){3}}
\linethickness{1.0pt}\put(48, 4){\line(0,1){4}}\put(46,6){\line(1,0){4}}
\linethickness{0.3pt}\put(48, 9){\line(0,1){3}}
\multiput(48,12)(0,2){6}{\put( 0,0){\line(0,1){1}}}
\multiput(48,24)(0,12){4}
{\linethickness{0.3pt}\put( 0,0){\line(0,1){3}}
 \linethickness{1.0pt}\put( 0,4){\line(0,1){4}}\put(-2,6){\line(1,0){4}} 
 \linethickness{0.3pt}\put( 0,9){\line(0,1){3}}
}
\multiput(48,72)(0,2){6}{\put( 0,0){\line(0,1){1}}}
\linethickness{0.3pt}\put(48,84){\line(0,1){3}}
\linethickness{1.0pt}\put(48,88){\line(0,1){4}} \put(46,90){\line(1,0){4}}
\linethickness{0.3pt}\put(48,93){\line(0,1){3}}
\end{picture}
\end{center}
with the initial condition $S_2(0) = \emptyset$ (the unfolded paper).
Here we use the notation $m'_a$ for the reflection in $x_a$-axis
direction. See Figure \ref{fig: S1S2S3} for a visualisation of $S_2$
after the first folds. Note that reflecting twice along different
axes, $m'_a m'_b$, results in just a rotation by $\pi$.
\begin{center}
\begin{figure}[th]
\begin{center}
\begin{tabular}{ccc}
\begin{picture}(96,96)
\linethickness{0.3pt}
\put(48,48){\line(0,1){21}}
\put(48,75){\line(0,1){21}}
\linethickness{1.0pt}
\put(48,70){\line(0,1){4}}
\put(46,72){\line(1,0){4}}
\linethickness{0.3pt}
\put(0,48){\line(1,0){21}}
\put(27,48){\line(1,0){21}}
\linethickness{1.0pt}
\put(22,48){\line(1,0){4}}
\linethickness{0.3pt}
\put(48,48){\line(1,0){21}}
\put(75,48){\line(1,0){21}}
\linethickness{1.0pt}
\put(70,48){\line(1,0){4}}
\put(72,46){\line(0,1){4}}
\linethickness{0.3pt}
\put(48,0){\line(0,1){21}}
\put(48,27){\line(0,1){21}}
\linethickness{1.0pt}
\put(48,22){\line(0,1){4}}
\put(46,24){\line(1,0){4}}
\end{picture}
& 
\begin{picture}(96,96)
\linethickness{0.3pt}
\put(24,72){\line(0,1){9}}
\put(24,87){\line(0,1){9}}
\linethickness{1.0pt}
\put(24,82){\line(0,1){4}}
\linethickness{0.3pt}
\put(48,72){\line(0,1){9}}
\put(48,87){\line(0,1){9}}
\linethickness{1.0pt}
\put(48,82){\line(0,1){4}}
\put(46,84){\line(1,0){4}}
\linethickness{0.3pt}
\put(72,72){\line(0,1){9}}
\put(72,87){\line(0,1){9}}
\linethickness{1.0pt}
\put(72,82){\line(0,1){4}}
\put(70,84){\line(1,0){4}}
\linethickness{0.3pt}
\put(0,72){\line(1,0){9}}
\put(15,72){\line(1,0){9}}
\linethickness{1.0pt}
\put(10,72){\line(1,0){4}}
\linethickness{0.3pt}
\put(24,72){\line(1,0){9}}
\put(39,72){\line(1,0){9}}
\linethickness{1.0pt}
\put(34,72){\line(1,0){4}}
\put(36,70){\line(0,1){4}}
\linethickness{0.3pt}
\put(48,72){\line(1,0){9}}
\put(63,72){\line(1,0){9}}
\linethickness{1.0pt}
\put(58,72){\line(1,0){4}}
\linethickness{0.3pt}
\put(72,72){\line(1,0){9}}
\put(87,72){\line(1,0){9}}
\linethickness{1.0pt}
\put(82,72){\line(1,0){4}}
\put(84,70){\line(0,1){4}}
\linethickness{0.3pt}
\put(24,48){\line(0,1){9}}
\put(24,63){\line(0,1){9}}
\linethickness{1.0pt}
\put(24,58){\line(0,1){4}}
\linethickness{0.3pt}
\put(48,48){\line(0,1){9}}
\put(48,63){\line(0,1){9}}
\linethickness{1.0pt}
\put(48,58){\line(0,1){4}}
\put(46,60){\line(1,0){4}}
\linethickness{0.3pt}
\put(72,48){\line(0,1){9}}
\put(72,63){\line(0,1){9}}
\linethickness{1.0pt}
\put(72,58){\line(0,1){4}}
\put(70,60){\line(1,0){4}}
\linethickness{0.3pt}
\put(0,48){\line(1,0){9}}
\put(15,48){\line(1,0){9}}
\linethickness{1.0pt}
\put(10,48){\line(1,0){4}}
\linethickness{0.3pt}
\put(24,48){\line(1,0){9}}
\put(39,48){\line(1,0){9}}
\linethickness{1.0pt}
\put(34,48){\line(1,0){4}}
\linethickness{0.3pt}
\put(48,48){\line(1,0){9}}
\put(63,48){\line(1,0){9}}
\linethickness{1.0pt}
\put(58,48){\line(1,0){4}}
\put(60,46){\line(0,1){4}}
\linethickness{0.3pt}
\put(72,48){\line(1,0){9}}
\put(87,48){\line(1,0){9}}
\linethickness{1.0pt}
\put(82,48){\line(1,0){4}}
\put(84,46){\line(0,1){4}}
\linethickness{0.3pt}
\put(24,24){\line(0,1){9}}
\put(24,39){\line(0,1){9}}
\linethickness{1.0pt}
\put(24,34){\line(0,1){4}}
\put(22,36){\line(1,0){4}}
\linethickness{0.3pt}
\put(48,24){\line(0,1){9}}
\put(48,39){\line(0,1){9}}
\linethickness{1.0pt}
\put(48,34){\line(0,1){4}}
\put(46,36){\line(1,0){4}}
\linethickness{0.3pt}
\put(72,24){\line(0,1){9}}
\put(72,39){\line(0,1){9}}
\linethickness{1.0pt}
\put(72,34){\line(0,1){4}}
\linethickness{0.3pt}
\put(0,24){\line(1,0){9}}
\put(15,24){\line(1,0){9}}
\linethickness{1.0pt}
\put(10,24){\line(1,0){4}}
\put(12,22){\line(0,1){4}}
\linethickness{0.3pt}
\put(24,24){\line(1,0){9}}
\put(39,24){\line(1,0){9}}
\linethickness{1.0pt}
\put(34,24){\line(1,0){4}}
\linethickness{0.3pt}
\put(48,24){\line(1,0){9}}
\put(63,24){\line(1,0){9}}
\linethickness{1.0pt}
\put(58,24){\line(1,0){4}}
\put(60,22){\line(0,1){4}}
\linethickness{0.3pt}
\put(72,24){\line(1,0){9}}
\put(87,24){\line(1,0){9}}
\linethickness{1.0pt}
\put(82,24){\line(1,0){4}}
\linethickness{0.3pt}
\put(24,0){\line(0,1){9}}
\put(24,15){\line(0,1){9}}
\linethickness{1.0pt}
\put(24,10){\line(0,1){4}}
\put(22,12){\line(1,0){4}}
\linethickness{0.3pt}
\put(48,0){\line(0,1){9}}
\put(48,15){\line(0,1){9}}
\linethickness{1.0pt}
\put(48,10){\line(0,1){4}}
\put(46,12){\line(1,0){4}}
\linethickness{0.3pt}
\put(72,0){\line(0,1){9}}
\put(72,15){\line(0,1){9}}
\linethickness{1.0pt}
\put(72,10){\line(0,1){4}}
\end{picture}
& 
\begin{picture}(96,96)
\linethickness{0.3pt}
\put(12,84){\line(0,1){3}}
\put(12,93){\line(0,1){3}}
\linethickness{1.0pt}
\put(12,88){\line(0,1){4}}
\linethickness{0.3pt}
\put(24,84){\line(0,1){3}}
\put(24,93){\line(0,1){3}}
\linethickness{1.0pt}
\put(24,88){\line(0,1){4}}
\linethickness{0.3pt}
\put(36,84){\line(0,1){3}}
\put(36,93){\line(0,1){3}}
\linethickness{1.0pt}
\put(36,88){\line(0,1){4}}
\put(34,90){\line(1,0){4}}
\linethickness{0.3pt}
\put(48,84){\line(0,1){3}}
\put(48,93){\line(0,1){3}}
\linethickness{1.0pt}
\put(48,88){\line(0,1){4}}
\put(46,90){\line(1,0){4}}
\linethickness{0.3pt}
\put(60,84){\line(0,1){3}}
\put(60,93){\line(0,1){3}}
\linethickness{1.0pt}
\put(60,88){\line(0,1){4}}
\linethickness{0.3pt}
\put(72,84){\line(0,1){3}}
\put(72,93){\line(0,1){3}}
\linethickness{1.0pt}
\put(72,88){\line(0,1){4}}
\put(70,90){\line(1,0){4}}
\linethickness{0.3pt}
\put(84,84){\line(0,1){3}}
\put(84,93){\line(0,1){3}}
\linethickness{1.0pt}
\put(84,88){\line(0,1){4}}
\put(82,90){\line(1,0){4}}
\linethickness{0.3pt}
\put(0,84){\line(1,0){3}}
\put(9,84){\line(1,0){3}}
\linethickness{1.0pt}
\put(4,84){\line(1,0){4}}
\linethickness{0.3pt}
\put(12,84){\line(1,0){3}}
\put(21,84){\line(1,0){3}}
\linethickness{1.0pt}
\put(16,84){\line(1,0){4}}
\put(18,82){\line(0,1){4}}
\linethickness{0.3pt}
\put(24,84){\line(1,0){3}}
\put(33,84){\line(1,0){3}}
\linethickness{1.0pt}
\put(28,84){\line(1,0){4}}
\linethickness{0.3pt}
\put(36,84){\line(1,0){3}}
\put(45,84){\line(1,0){3}}
\linethickness{1.0pt}
\put(40,84){\line(1,0){4}}
\put(42,82){\line(0,1){4}}
\linethickness{0.3pt}
\put(48,84){\line(1,0){3}}
\put(57,84){\line(1,0){3}}
\linethickness{1.0pt}
\put(52,84){\line(1,0){4}}
\linethickness{0.3pt}
\put(60,84){\line(1,0){3}}
\put(69,84){\line(1,0){3}}
\linethickness{1.0pt}
\put(64,84){\line(1,0){4}}
\put(66,82){\line(0,1){4}}
\linethickness{0.3pt}
\put(72,84){\line(1,0){3}}
\put(81,84){\line(1,0){3}}
\linethickness{1.0pt}
\put(76,84){\line(1,0){4}}
\linethickness{0.3pt}
\put(84,84){\line(1,0){3}}
\put(93,84){\line(1,0){3}}
\linethickness{1.0pt}
\put(88,84){\line(1,0){4}}
\put(90,82){\line(0,1){4}}
\linethickness{0.3pt}
\put(12,72){\line(0,1){3}}
\put(12,81){\line(0,1){3}}
\linethickness{1.0pt}
\put(12,76){\line(0,1){4}}
\linethickness{0.3pt}
\put(24,72){\line(0,1){3}}
\put(24,81){\line(0,1){3}}
\linethickness{1.0pt}
\put(24,76){\line(0,1){4}}
\linethickness{0.3pt}
\put(36,72){\line(0,1){3}}
\put(36,81){\line(0,1){3}}
\linethickness{1.0pt}
\put(36,76){\line(0,1){4}}
\put(34,78){\line(1,0){4}}
\linethickness{0.3pt}
\put(48,72){\line(0,1){3}}
\put(48,81){\line(0,1){3}}
\linethickness{1.0pt}
\put(48,76){\line(0,1){4}}
\put(46,78){\line(1,0){4}}
\linethickness{0.3pt}
\put(60,72){\line(0,1){3}}
\put(60,81){\line(0,1){3}}
\linethickness{1.0pt}
\put(60,76){\line(0,1){4}}
\linethickness{0.3pt}
\put(72,72){\line(0,1){3}}
\put(72,81){\line(0,1){3}}
\linethickness{1.0pt}
\put(72,76){\line(0,1){4}}
\put(70,78){\line(1,0){4}}
\linethickness{0.3pt}
\put(84,72){\line(0,1){3}}
\put(84,81){\line(0,1){3}}
\linethickness{1.0pt}
\put(84,76){\line(0,1){4}}
\put(82,78){\line(1,0){4}}
\linethickness{0.3pt}
\put(0,72){\line(1,0){3}}
\put(9,72){\line(1,0){3}}
\linethickness{1.0pt}
\put(4,72){\line(1,0){4}}
\linethickness{0.3pt}
\put(12,72){\line(1,0){3}}
\put(21,72){\line(1,0){3}}
\linethickness{1.0pt}
\put(16,72){\line(1,0){4}}
\linethickness{0.3pt}
\put(24,72){\line(1,0){3}}
\put(33,72){\line(1,0){3}}
\linethickness{1.0pt}
\put(28,72){\line(1,0){4}}
\put(30,70){\line(0,1){4}}
\linethickness{0.3pt}
\put(36,72){\line(1,0){3}}
\put(45,72){\line(1,0){3}}
\linethickness{1.0pt}
\put(40,72){\line(1,0){4}}
\put(42,70){\line(0,1){4}}
\linethickness{0.3pt}
\put(48,72){\line(1,0){3}}
\put(57,72){\line(1,0){3}}
\linethickness{1.0pt}
\put(52,72){\line(1,0){4}}
\linethickness{0.3pt}
\put(60,72){\line(1,0){3}}
\put(69,72){\line(1,0){3}}
\linethickness{1.0pt}
\put(64,72){\line(1,0){4}}
\linethickness{0.3pt}
\put(72,72){\line(1,0){3}}
\put(81,72){\line(1,0){3}}
\linethickness{1.0pt}
\put(76,72){\line(1,0){4}}
\put(78,70){\line(0,1){4}}
\linethickness{0.3pt}
\put(84,72){\line(1,0){3}}
\put(93,72){\line(1,0){3}}
\linethickness{1.0pt}
\put(88,72){\line(1,0){4}}
\put(90,70){\line(0,1){4}}
\linethickness{0.3pt}
\put(12,60){\line(0,1){3}}
\put(12,69){\line(0,1){3}}
\linethickness{1.0pt}
\put(12,64){\line(0,1){4}}
\put(10,66){\line(1,0){4}}
\linethickness{0.3pt}
\put(24,60){\line(0,1){3}}
\put(24,69){\line(0,1){3}}
\linethickness{1.0pt}
\put(24,64){\line(0,1){4}}
\linethickness{0.3pt}
\put(36,60){\line(0,1){3}}
\put(36,69){\line(0,1){3}}
\linethickness{1.0pt}
\put(36,64){\line(0,1){4}}
\linethickness{0.3pt}
\put(48,60){\line(0,1){3}}
\put(48,69){\line(0,1){3}}
\linethickness{1.0pt}
\put(48,64){\line(0,1){4}}
\put(46,66){\line(1,0){4}}
\linethickness{0.3pt}
\put(60,60){\line(0,1){3}}
\put(60,69){\line(0,1){3}}
\linethickness{1.0pt}
\put(60,64){\line(0,1){4}}
\put(58,66){\line(1,0){4}}
\linethickness{0.3pt}
\put(72,60){\line(0,1){3}}
\put(72,69){\line(0,1){3}}
\linethickness{1.0pt}
\put(72,64){\line(0,1){4}}
\put(70,66){\line(1,0){4}}
\linethickness{0.3pt}
\put(84,60){\line(0,1){3}}
\put(84,69){\line(0,1){3}}
\linethickness{1.0pt}
\put(84,64){\line(0,1){4}}
\linethickness{0.3pt}
\put(0,60){\line(1,0){3}}
\put(9,60){\line(1,0){3}}
\linethickness{1.0pt}
\put(4,60){\line(1,0){4}}
\put(6,58){\line(0,1){4}}
\linethickness{0.3pt}
\put(12,60){\line(1,0){3}}
\put(21,60){\line(1,0){3}}
\linethickness{1.0pt}
\put(16,60){\line(1,0){4}}
\linethickness{0.3pt}
\put(24,60){\line(1,0){3}}
\put(33,60){\line(1,0){3}}
\linethickness{1.0pt}
\put(28,60){\line(1,0){4}}
\put(30,58){\line(0,1){4}}
\linethickness{0.3pt}
\put(36,60){\line(1,0){3}}
\put(45,60){\line(1,0){3}}
\linethickness{1.0pt}
\put(40,60){\line(1,0){4}}
\linethickness{0.3pt}
\put(48,60){\line(1,0){3}}
\put(57,60){\line(1,0){3}}
\linethickness{1.0pt}
\put(52,60){\line(1,0){4}}
\put(54,58){\line(0,1){4}}
\linethickness{0.3pt}
\put(60,60){\line(1,0){3}}
\put(69,60){\line(1,0){3}}
\linethickness{1.0pt}
\put(64,60){\line(1,0){4}}
\linethickness{0.3pt}
\put(72,60){\line(1,0){3}}
\put(81,60){\line(1,0){3}}
\linethickness{1.0pt}
\put(76,60){\line(1,0){4}}
\put(78,58){\line(0,1){4}}
\linethickness{0.3pt}
\put(84,60){\line(1,0){3}}
\put(93,60){\line(1,0){3}}
\linethickness{1.0pt}
\put(88,60){\line(1,0){4}}
\linethickness{0.3pt}
\put(12,48){\line(0,1){3}}
\put(12,57){\line(0,1){3}}
\linethickness{1.0pt}
\put(12,52){\line(0,1){4}}
\put(10,54){\line(1,0){4}}
\linethickness{0.3pt}
\put(24,48){\line(0,1){3}}
\put(24,57){\line(0,1){3}}
\linethickness{1.0pt}
\put(24,52){\line(0,1){4}}
\linethickness{0.3pt}
\put(36,48){\line(0,1){3}}
\put(36,57){\line(0,1){3}}
\linethickness{1.0pt}
\put(36,52){\line(0,1){4}}
\linethickness{0.3pt}
\put(48,48){\line(0,1){3}}
\put(48,57){\line(0,1){3}}
\linethickness{1.0pt}
\put(48,52){\line(0,1){4}}
\put(46,54){\line(1,0){4}}
\linethickness{0.3pt}
\put(60,48){\line(0,1){3}}
\put(60,57){\line(0,1){3}}
\linethickness{1.0pt}
\put(60,52){\line(0,1){4}}
\put(58,54){\line(1,0){4}}
\linethickness{0.3pt}
\put(72,48){\line(0,1){3}}
\put(72,57){\line(0,1){3}}
\linethickness{1.0pt}
\put(72,52){\line(0,1){4}}
\put(70,54){\line(1,0){4}}
\linethickness{0.3pt}
\put(84,48){\line(0,1){3}}
\put(84,57){\line(0,1){3}}
\linethickness{1.0pt}
\put(84,52){\line(0,1){4}}
\linethickness{0.3pt}
\put(0,48){\line(1,0){3}}
\put(9,48){\line(1,0){3}}
\linethickness{1.0pt}
\put(4,48){\line(1,0){4}}
\linethickness{0.3pt}
\put(12,48){\line(1,0){3}}
\put(21,48){\line(1,0){3}}
\linethickness{1.0pt}
\put(16,48){\line(1,0){4}}
\linethickness{0.3pt}
\put(24,48){\line(1,0){3}}
\put(33,48){\line(1,0){3}}
\linethickness{1.0pt}
\put(28,48){\line(1,0){4}}
\linethickness{0.3pt}
\put(36,48){\line(1,0){3}}
\put(45,48){\line(1,0){3}}
\linethickness{1.0pt}
\put(40,48){\line(1,0){4}}
\linethickness{0.3pt}
\put(48,48){\line(1,0){3}}
\put(57,48){\line(1,0){3}}
\linethickness{1.0pt}
\put(52,48){\line(1,0){4}}
\put(54,46){\line(0,1){4}}
\linethickness{0.3pt}
\put(60,48){\line(1,0){3}}
\put(69,48){\line(1,0){3}}
\linethickness{1.0pt}
\put(64,48){\line(1,0){4}}
\put(66,46){\line(0,1){4}}
\linethickness{0.3pt}
\put(72,48){\line(1,0){3}}
\put(81,48){\line(1,0){3}}
\linethickness{1.0pt}
\put(76,48){\line(1,0){4}}
\put(78,46){\line(0,1){4}}
\linethickness{0.3pt}
\put(84,48){\line(1,0){3}}
\put(93,48){\line(1,0){3}}
\linethickness{1.0pt}
\put(88,48){\line(1,0){4}}
\put(90,46){\line(0,1){4}}
\linethickness{0.3pt}
\put(12,36){\line(0,1){3}}
\put(12,45){\line(0,1){3}}
\linethickness{1.0pt}
\put(12,40){\line(0,1){4}}
\linethickness{0.3pt}
\put(24,36){\line(0,1){3}}
\put(24,45){\line(0,1){3}}
\linethickness{1.0pt}
\put(24,40){\line(0,1){4}}
\put(22,42){\line(1,0){4}}
\linethickness{0.3pt}
\put(36,36){\line(0,1){3}}
\put(36,45){\line(0,1){3}}
\linethickness{1.0pt}
\put(36,40){\line(0,1){4}}
\put(34,42){\line(1,0){4}}
\linethickness{0.3pt}
\put(48,36){\line(0,1){3}}
\put(48,45){\line(0,1){3}}
\linethickness{1.0pt}
\put(48,40){\line(0,1){4}}
\put(46,42){\line(1,0){4}}
\linethickness{0.3pt}
\put(60,36){\line(0,1){3}}
\put(60,45){\line(0,1){3}}
\linethickness{1.0pt}
\put(60,40){\line(0,1){4}}
\linethickness{0.3pt}
\put(72,36){\line(0,1){3}}
\put(72,45){\line(0,1){3}}
\linethickness{1.0pt}
\put(72,40){\line(0,1){4}}
\linethickness{0.3pt}
\put(84,36){\line(0,1){3}}
\put(84,45){\line(0,1){3}}
\linethickness{1.0pt}
\put(84,40){\line(0,1){4}}
\put(82,42){\line(1,0){4}}
\linethickness{0.3pt}
\put(0,36){\line(1,0){3}}
\put(9,36){\line(1,0){3}}
\linethickness{1.0pt}
\put(4,36){\line(1,0){4}}
\linethickness{0.3pt}
\put(12,36){\line(1,0){3}}
\put(21,36){\line(1,0){3}}
\linethickness{1.0pt}
\put(16,36){\line(1,0){4}}
\put(18,34){\line(0,1){4}}
\linethickness{0.3pt}
\put(24,36){\line(1,0){3}}
\put(33,36){\line(1,0){3}}
\linethickness{1.0pt}
\put(28,36){\line(1,0){4}}
\linethickness{0.3pt}
\put(36,36){\line(1,0){3}}
\put(45,36){\line(1,0){3}}
\linethickness{1.0pt}
\put(40,36){\line(1,0){4}}
\put(42,34){\line(0,1){4}}
\linethickness{0.3pt}
\put(48,36){\line(1,0){3}}
\put(57,36){\line(1,0){3}}
\linethickness{1.0pt}
\put(52,36){\line(1,0){4}}
\linethickness{0.3pt}
\put(60,36){\line(1,0){3}}
\put(69,36){\line(1,0){3}}
\linethickness{1.0pt}
\put(64,36){\line(1,0){4}}
\put(66,34){\line(0,1){4}}
\linethickness{0.3pt}
\put(72,36){\line(1,0){3}}
\put(81,36){\line(1,0){3}}
\linethickness{1.0pt}
\put(76,36){\line(1,0){4}}
\linethickness{0.3pt}
\put(84,36){\line(1,0){3}}
\put(93,36){\line(1,0){3}}
\linethickness{1.0pt}
\put(88,36){\line(1,0){4}}
\put(90,34){\line(0,1){4}}
\linethickness{0.3pt}
\put(12,24){\line(0,1){3}}
\put(12,33){\line(0,1){3}}
\linethickness{1.0pt}
\put(12,28){\line(0,1){4}}
\linethickness{0.3pt}
\put(24,24){\line(0,1){3}}
\put(24,33){\line(0,1){3}}
\linethickness{1.0pt}
\put(24,28){\line(0,1){4}}
\put(22,30){\line(1,0){4}}
\linethickness{0.3pt}
\put(36,24){\line(0,1){3}}
\put(36,33){\line(0,1){3}}
\linethickness{1.0pt}
\put(36,28){\line(0,1){4}}
\put(34,30){\line(1,0){4}}
\linethickness{0.3pt}
\put(48,24){\line(0,1){3}}
\put(48,33){\line(0,1){3}}
\linethickness{1.0pt}
\put(48,28){\line(0,1){4}}
\put(46,30){\line(1,0){4}}
\linethickness{0.3pt}
\put(60,24){\line(0,1){3}}
\put(60,33){\line(0,1){3}}
\linethickness{1.0pt}
\put(60,28){\line(0,1){4}}
\linethickness{0.3pt}
\put(72,24){\line(0,1){3}}
\put(72,33){\line(0,1){3}}
\linethickness{1.0pt}
\put(72,28){\line(0,1){4}}
\linethickness{0.3pt}
\put(84,24){\line(0,1){3}}
\put(84,33){\line(0,1){3}}
\linethickness{1.0pt}
\put(84,28){\line(0,1){4}}
\put(82,30){\line(1,0){4}}
\linethickness{0.3pt}
\put(0,24){\line(1,0){3}}
\put(9,24){\line(1,0){3}}
\linethickness{1.0pt}
\put(4,24){\line(1,0){4}}
\put(6,22){\line(0,1){4}}
\linethickness{0.3pt}
\put(12,24){\line(1,0){3}}
\put(21,24){\line(1,0){3}}
\linethickness{1.0pt}
\put(16,24){\line(1,0){4}}
\put(18,22){\line(0,1){4}}
\linethickness{0.3pt}
\put(24,24){\line(1,0){3}}
\put(33,24){\line(1,0){3}}
\linethickness{1.0pt}
\put(28,24){\line(1,0){4}}
\linethickness{0.3pt}
\put(36,24){\line(1,0){3}}
\put(45,24){\line(1,0){3}}
\linethickness{1.0pt}
\put(40,24){\line(1,0){4}}
\linethickness{0.3pt}
\put(48,24){\line(1,0){3}}
\put(57,24){\line(1,0){3}}
\linethickness{1.0pt}
\put(52,24){\line(1,0){4}}
\put(54,22){\line(0,1){4}}
\linethickness{0.3pt}
\put(60,24){\line(1,0){3}}
\put(69,24){\line(1,0){3}}
\linethickness{1.0pt}
\put(64,24){\line(1,0){4}}
\put(66,22){\line(0,1){4}}
\linethickness{0.3pt}
\put(72,24){\line(1,0){3}}
\put(81,24){\line(1,0){3}}
\linethickness{1.0pt}
\put(76,24){\line(1,0){4}}
\linethickness{0.3pt}
\put(84,24){\line(1,0){3}}
\put(93,24){\line(1,0){3}}
\linethickness{1.0pt}
\put(88,24){\line(1,0){4}}
\linethickness{0.3pt}
\put(12,12){\line(0,1){3}}
\put(12,21){\line(0,1){3}}
\linethickness{1.0pt}
\put(12,16){\line(0,1){4}}
\put(10,18){\line(1,0){4}}
\linethickness{0.3pt}
\put(24,12){\line(0,1){3}}
\put(24,21){\line(0,1){3}}
\linethickness{1.0pt}
\put(24,16){\line(0,1){4}}
\put(22,18){\line(1,0){4}}
\linethickness{0.3pt}
\put(36,12){\line(0,1){3}}
\put(36,21){\line(0,1){3}}
\linethickness{1.0pt}
\put(36,16){\line(0,1){4}}
\linethickness{0.3pt}
\put(48,12){\line(0,1){3}}
\put(48,21){\line(0,1){3}}
\linethickness{1.0pt}
\put(48,16){\line(0,1){4}}
\put(46,18){\line(1,0){4}}
\linethickness{0.3pt}
\put(60,12){\line(0,1){3}}
\put(60,21){\line(0,1){3}}
\linethickness{1.0pt}
\put(60,16){\line(0,1){4}}
\put(58,18){\line(1,0){4}}
\linethickness{0.3pt}
\put(72,12){\line(0,1){3}}
\put(72,21){\line(0,1){3}}
\linethickness{1.0pt}
\put(72,16){\line(0,1){4}}
\linethickness{0.3pt}
\put(84,12){\line(0,1){3}}
\put(84,21){\line(0,1){3}}
\linethickness{1.0pt}
\put(84,16){\line(0,1){4}}
\linethickness{0.3pt}
\put(0,12){\line(1,0){3}}
\put(9,12){\line(1,0){3}}
\linethickness{1.0pt}
\put(4,12){\line(1,0){4}}
\put(6,10){\line(0,1){4}}
\linethickness{0.3pt}
\put(12,12){\line(1,0){3}}
\put(21,12){\line(1,0){3}}
\linethickness{1.0pt}
\put(16,12){\line(1,0){4}}
\linethickness{0.3pt}
\put(24,12){\line(1,0){3}}
\put(33,12){\line(1,0){3}}
\linethickness{1.0pt}
\put(28,12){\line(1,0){4}}
\put(30,10){\line(0,1){4}}
\linethickness{0.3pt}
\put(36,12){\line(1,0){3}}
\put(45,12){\line(1,0){3}}
\linethickness{1.0pt}
\put(40,12){\line(1,0){4}}
\linethickness{0.3pt}
\put(48,12){\line(1,0){3}}
\put(57,12){\line(1,0){3}}
\linethickness{1.0pt}
\put(52,12){\line(1,0){4}}
\put(54,10){\line(0,1){4}}
\linethickness{0.3pt}
\put(60,12){\line(1,0){3}}
\put(69,12){\line(1,0){3}}
\linethickness{1.0pt}
\put(64,12){\line(1,0){4}}
\linethickness{0.3pt}
\put(72,12){\line(1,0){3}}
\put(81,12){\line(1,0){3}}
\linethickness{1.0pt}
\put(76,12){\line(1,0){4}}
\put(78,10){\line(0,1){4}}
\linethickness{0.3pt}
\put(84,12){\line(1,0){3}}
\put(93,12){\line(1,0){3}}
\linethickness{1.0pt}
\put(88,12){\line(1,0){4}}
\linethickness{0.3pt}
\put(12,0){\line(0,1){3}}
\put(12,9){\line(0,1){3}}
\linethickness{1.0pt}
\put(12,4){\line(0,1){4}}
\put(10,6){\line(1,0){4}}
\linethickness{0.3pt}
\put(24,0){\line(0,1){3}}
\put(24,9){\line(0,1){3}}
\linethickness{1.0pt}
\put(24,4){\line(0,1){4}}
\put(22,6){\line(1,0){4}}
\linethickness{0.3pt}
\put(36,0){\line(0,1){3}}
\put(36,9){\line(0,1){3}}
\linethickness{1.0pt}
\put(36,4){\line(0,1){4}}
\linethickness{0.3pt}
\put(48,0){\line(0,1){3}}
\put(48,9){\line(0,1){3}}
\linethickness{1.0pt}
\put(48,4){\line(0,1){4}}
\put(46,6){\line(1,0){4}}
\linethickness{0.3pt}
\put(60,0){\line(0,1){3}}
\put(60,9){\line(0,1){3}}
\linethickness{1.0pt}
\put(60,4){\line(0,1){4}}
\put(58,6){\line(1,0){4}}
\linethickness{0.3pt}
\put(72,0){\line(0,1){3}}
\put(72,9){\line(0,1){3}}
\linethickness{1.0pt}
\put(72,4){\line(0,1){4}}
\linethickness{0.3pt}
\put(84,0){\line(0,1){3}}
\put(84,9){\line(0,1){3}}
\linethickness{1.0pt}
\put(84,4){\line(0,1){4}}
\end{picture}
\\[2ex]
{\small$S_2(1)$} & {\small$S_2(2)$} & {\small$S_2(3)$}
	\end{tabular}
\end{center}
\caption{\label{fig: S1S2S3} 
  The first three generations of the 2-dimensional paperfolding structure.}
\end{figure}
\end{center}

For the general $d$-dimensional recursion we need to describe the
first $d$-fold, $S_d(1)$. The creases in $S_d(1)$ can be labelled by
vectors $\boldsymbol{\sigma}=(\sigma_1,\ldots,\sigma_d)$ $\in \{-1,0,1\}^d$,
which have precisely one component 0. We denote the set of admissible
crease labels $\boldsymbol{\sigma}$ by $C_d$. The position $k$ of the 
component $0$ identifies the orientation of the hyperplane containing 
the fold (perpendicular to the $k$-axis), and the remaining components 
$\sigma_i=\pm1$ together specify one of $2^{d-1}$ sectors within that 
hyperplane, via the conditions $\sigma_ix_i>0$. In each such sector,
the crease sign is constant. Therefore, to each crease label 
$\boldsymbol{\sigma}\in C_d$ there corresponds a unique crease sign, 
which we denote by $c(\boldsymbol{\sigma})$. In order to simplify the 
notation, in the following we denote by $\boldsymbol{\sigma}^k$ any 
crease label having the single component $0$ at position $k$, and by
$\sigma^k_i$ its $i$-th component.

The first 1-fold of the paper, along $x_1=0$, gives rise to the creases 
labelled by $\boldsymbol{\sigma}^1$, and all of them have the same sign 
$+$, so that
\[	c(\boldsymbol{\sigma}^1) = +,
\]
whatever the components $\sigma^1_i$ are.
The second $1$-fold of the paper, along $x_2=0$, gives rise to the
creases $\boldsymbol{\sigma}^2$. Here we have to notice that the
creases in the region $x_1<0$ must have the opposite sign. This gives
\[	c(\boldsymbol{\sigma}^2) = 
  \begin{cases}
    \mathtt{+} & \text{if } \sigma^2_1>0 \\
    \mathtt{-} & \text{if } \sigma^2_1<0 
  \end{cases}
\]
Continuing this way, we see that when making the $1$-fold along
$x_j=0$, we have to take into account that all creases in the region
$x_{j-1}<0$ have to have the opposite sign. This leads to the
recursion
\[	c(\boldsymbol{\sigma}^j) = 
  \begin{cases}
    \phantom{-} c(\boldsymbol{\sigma}^{j-1}) & \text{if } \sigma^j_{j-1}>0, \\
    - c(\boldsymbol{\sigma}^{j-1}) & \text{if } \sigma^j_{j-1}<0, 
  \end{cases}
\]
where $-c$ denotes the opposite sign of $c$, $\sigma^{j-1}_i=\sigma^j_i$ 
for $i<j-1$, and $\sigma^{j-1}_i=1$ for $i>j-1$. By induction it now 
easily follows that
\begin{equation}
  c(\boldsymbol{\sigma}^k) 
  = \text{sign} \Big(\prod_{i<k} \sigma^k_{i}\Big)
  \label{eq:  prod def of c}
\end{equation}
The collection of the creases given by the function $c$ on $C_d$ now 
allows to make the identification
\[
  S_d(1) \quad \leftrightarrow \quad  
  \bigcup_{\boldsymbol{\sigma} \in C_d} c(\boldsymbol{\sigma}).
\]
We will return to the properties of the function $c$ and the sign of
the creases of $S_d(1)$ in Section \ref{sec:properties}.

To describe $S_d(n+1)$, we start from $S_d(1)$ in the centre, and in
each of the $2^d$ orthants we place a reflected copy of $S_d(n)$. The
reflections can be described by
\[	M(\boldsymbol\phi) = \prod_{\{i\,:\, \phi_i = -1\}} m_i'
\]
where $\boldsymbol\phi = (\phi_1, \phi_2,\ldots,\phi_d)\in\{1,-1\}^d$
labels the orthants $\phi_i x_i>0$. We can now formulate the general 
recursion by the identification
\begin{equation}
\label{eq: gen rec}
S_d(n+1) \quad \leftrightarrow\quad 
S_d(1) \cup \Big(\bigcup_{\boldsymbol\phi\in\{\pm1\}^d} M(\boldsymbol\phi) S_d(n)\Big)
\end{equation}
with the initial condition $S_d(0)= \emptyset$.

\section{Substitution rule}
\label{sec:subst}

We present here a general scheme for creating a substitution for the
$d$-dimensional paperfolding structure. Recall that in the recursion
described above, a finite paperfolding pattern is extended by
appending reflected copies of itself. A substitution works differently.
Here, the extra creases that are inserted by the next $d$-fold 
in the interior of a cube are determined locally. A cube is first 
expanded by a factor 2, and then the new creases are inserted 
so that it gets divided into a block of $2^d$ new cubes, which have
again the original size. 

As it turns out, the creases that are inserted into a cube depend not
so much on the local crease pattern around the cube, but rather on the
parity of the position of the cube. We begin by discussing this in
dimension 1. The generalisation to higher dimensions is then immediate.

\begin{lemma}
\label{lemma: one dim s}
Let $s$ be a unit interval, with left end point in $x$, in the crease
pattern of $S_1(n)$, for $n\geq 2$. When folding the paper $n$ times
together, $s$ is facing downwards in the pile if and only if $x$ is
even.
\end{lemma}

\begin{proof}
We give a proof by induction on $n$. By straight forward inspection,
it is easy to check that the Lemma holds for $n=2$. Suppose now that
it holds for some $n\ge2$, and let $s'$ be an interval of length $2$
in the crease pattern $S_1(n+1)$, with a left end point $x$ that is 
an even integer. When folding up $n$ times, we get a pile of intervals
of length $2$. If $x$ is twice an odd integer, $s'$ has its face up on
that pile. When folding the pile once more, the left half of the
pile, including the left half of $s'$, originally located at the
even integer $x$, gets reflected once more, whereas the right half
of $s'$, originally located at the odd position $x+1$, remains as
it is. In this case, the Lemma thus holds also for $n+1$. On the
other hand, if $x$ is twice an even integer, $s'$ is face down on
the pile, and the left and right halves of $s'$ have reversed order
on the pile. So, the right half of $s'$ now gets reflected again,
but not the left half, so that the Lemma holds for $n+1$ also in 
this case.
\end{proof}

In the same spirit as in Lemma \ref{lemma: one dim s} we can argue that 
a similar result holds in dimension $d$. For this generalisation, let
$s$ be a unit cube in the (unfolded) crease pattern of $S_d(n)$, 
$n\geq 2$. We associate a reference point $\boldsymbol{x}\in\mathbb{Z}^d$ 
to $s$, as the point in $s$ with the smallest coordinate values in all 
coordinate directions. To each cube $s$, we attach an orientation vector
$v\in\{1,-1\}^d$ pointing from the reference point of $s$ to the opposite
corner. When we talk about the orientation of $s$, we actually mean
its orientation vector. In the unfolded pattern, all orientations
are the same. However, reference point and orientation vector are
attached to the cube. When folding up the paper, cubes get reflected,
and at the same time, their reference points and orientation vectors 
are mapped to their mirror images. By the notation $m_{i}'s$ we mean 
the cube $s$ reflected along the $x_i$-axis. During this reflection, 
its orientation vector gets a sign change in its $i$-th component.
It is not hard to see that the reflections fulfil
\begin{equation}
\label{eq: m prop}
m_{i}'m_{j}'s = m_{j}'m_{i}'s \qquad \textnormal{and} \qquad m_{i}'m_{i}'s = s. 
\end{equation}
In other words, the reflections $m_i'$ are involutions and commute.
As we are interested only in the orientation of $s$, when folding the
paper together, it is enough to look at the chain of reflections
$m_{a_1}'m_{a_2}'\ldots m_{a_k}'$ due to the $n$ $d$-folds, and keep track
of the orientation of $s$. By the properties (\ref{eq: m prop}), 
this chain can be reduced to a chain of at most $d$ reflections, 
and with at most one reflection along each axis. The commutativity 
of the reflections implies that we may consider them as a product of 
1-dimensional reflections. By Lemma \ref{lemma: one dim s} it now 
follows that the resulting orientation of $s$ is only dependent on 
the parity of the coordinates in $\boldsymbol{x}$, as we can consider 
one axis at a time, and the $i$-th coordinate $\boldsymbol{x}$ only 
affects the sign of the $i$-th coordinate in $v$. To summarise, 
we have the following result.

\begin{lemma}
\label{lemma: d dim s}
Let $s$ be a unit cube in the crease pattern of $S_d(n)$, with
$n\ge2$ and $d\ge1$. When folding the paper $n$ times together, 
the orientation of $s$ only depends on the parity of the coordinates 
of the reference point $\boldsymbol{x}\in \mathbb{Z}^d$ of $s$.
\end{lemma}

An immediate consequence of Lemma \ref{lemma: d dim s} is that for a
unit cube $s$ in the crease pattern of $S_d(n)$, the creases that
will occur inside $s$, when folding the paper one additional time, do
not depend on the bounding creases of $s$, nor on the general position
of $s$, but only on the parity of the coordinates of its reference
point $\boldsymbol{x}$. Moreover, the crease pattern that will occur
inside $s$ is just a scaled and reflected copy of $S_d(1)$. Its 
precise orientation depends, again, only on the parity of the
coordinates of $\boldsymbol{x}$. As we can regard a $d$-fold as a
sequence of $d$ $1$-folds, Lemma \ref{lemma: one dim s} implies
the following result.

\begin{lemma}
\label{lemma: ori. of s}
Let $s$ be a unit cube in the crease pattern of $S_d(n)$, with
$n\ge2$ and $d\ge1$, and let $\boldsymbol{x}\in \mathbb{Z}^d$ 
be the reference point of $s$. When folding the paper $n+1$ times, 
the folds that will appear inside $s$ are
\begin{equation}
\label{eq: ori. of s}
\Big(\prod_{\{i\,:\, x_i \text{ even} \}}m_i'\Big)\, S_d(1). 
\end{equation}
\end{lemma}

From here on we can now argue how to construct a substitution rule,
$\mu_d$, that generates the $d$-dimensional paperfolding structure. In
Section \ref{sec:examples} we shall illustrate this construction by 
concrete examples.
The crease pattern of a paperfolding structure divides space into
unit cubes, whose faces contain the creases, and where each crease
has a sign (valley or crest). In order to attribute each crease
to a unique cube, we introduce the concept of a \emph{semi-cube},
which is a half-open cube, containing only one of each pair of 
parallel faces, namely the one with smaller coordinates in the 
direction perpendicular to the face. These are the faces which 
contain the reference point of the cube. Each of these $d$ faces 
may contain a crease of either sign, so that we have at most 
$2^d$ crease configurations on a semi-cube. Furthermore, the 
reference point of $s$ may have $2^d$ different parity values, 
on which the fold introduced in the interior of the cube upon 
substitution will depend. Taking all together, we need at most
$2^d \cdot 2^d$ tile types to define a substitution, where our
tiles are semi-cubes carrying parity information and a crease 
pattern on their faces. Upon substitution, each semi-cube $s$ 
is now mapped to a block of $2^d$ semi-cubes. The creases on the
outer faces of the block are just copied from the faces of $s$
with the same orientation, whereas the faces in the interior
are a reflected copy of $S_d(1)$, as given by (\ref{eq: ori. of s}).

In the spirit of the construction by recursion, we can write the
substitution as
\begin{equation}
\label{eq: gen sub}
\mu_d\,:\,s'\mapsto 2\cdot s'\cup 
  \Big(\prod_{\{i\,:\, x_i \text{ even} \}}m_i'\Big)\,S_d(1) 
\end{equation}
where $s'$ is a semi-cube with the reference point $\boldsymbol{x}$,
and $2\cdot s'$ is the semi-cube scaled by a factor $2$.  As a
seed for generating the $d$-dimensional paperfolding structure with
the help of $\mu_d$, we take the central $2^d$-block of semi-cubes
in $S_d(2)$. Note that the substitution $\mu_d$ generates the 
paperfolding structure from the centre outwards.

\section{Examples}
\label{sec:examples}

Here we illustrate the general construction of the paperfolding
substitution presented in Section \ref{sec:subst} by giving a
more concrete, detailed representation for the 1- and 2-dimensional 
cases. In particular, we give a visual representation of the 
the map given in (\ref{eq: gen sub}).

Let us start with the 1-dimensional case, for which a substitution 
is already well known \cite{allouche95, dekking82}. We discuss
it here to illustrate the general construction. Any tile corresponds
to a semi-cube $s$ in the crease pattern of $S_1(n)$, that is a 
unit interval, with reference point $\boldsymbol{x} = (x_1)$ at the left 
end point. The right end point is removed, so that we have a half-open
unit interval. Upon substitution, each such semi-cube is doubled in size, 
with left end point unaffected by the substitution, and with
a reflected copy of $S_1(1)$ added to the interior, according to 
Lemma~\ref{lemma: ori. of s}. Hence, we may define $\mu_1$ as
follows:
\small
\begin{equation}
\begin{array}{r}
\begin{tabular}{m{12pt}c*{2}{m{36pt}}}
& & $x_1$ even & $x_1$ odd \\ 
\end{tabular}
\\[3ex]
\mu_1 : 
\left\{
\begin{tabular}{m{4pt}c*{2}{m{36pt}}@{}}
\plusline  & $\mapsto$ & \plusline\minusline  & \plusline\plusline  \\ \\
\minusline & $\mapsto$ & \minusline\minusline & \minusline\plusline
\end{tabular}
\right.
\end{array}
\label{s1mu}
\end{equation}
\normalsize 
For the seed we take the pattern \minus \plus, which covers the half-open 
interval $[-1,1)$, where the \plus crease is at the centre. 
The substitution $\mu_1$ can also be written as a 
substitution on the 4 letter alphabet $A = \{a_{ij}: 0\leq i,j\leq1 \}$, 
where the first index indicates the type of crease at the left end point
of the semi-cube, and the second index the parity of its position.  
The assignment of the indices is as given in the rows, resp.\ columns
of the table in Eq.~(\ref{s1mu}). Symbolically, the substitution can thus be 
written as
\small
\[
\nu_1 : \left\{
\begin{tabular}{@{}*{2}{r@{\,$\mapsto$\,}l}}
$a_{00}$ & $a_{00}$ $ a_{11}$, &
$a_{01}$ & $a_{00}$ $ a_{01}$,\\[1ex] 
$a_{10}$ & $a_{10}$ $ a_{11}$,& 
$a_{11}$ & $a_{10}$ $ a_{01}$,
\end{tabular}
\right.
\label{s1nu}
\]
\normalsize
with seed $a_{11}a_{00}$.

Similarly, we can define a substitution $\mu_2$ for the
2-dimensional paperfolding structure. Again, we start by considering a
unit square $s$ in the crease pattern of $S_2(n)$, with lower left
corner at $\boldsymbol{x}= (x_1,x_2)$. The crease pattern in the interior of
the inflated tiles follows from Lemma \ref{lemma: ori. of s}, resulting
in a substitution $\mu_2$ as follows:

\small
\begin{equation}
\begin{array}{r}
\begin{tabular}{m{12pt}c*{4}{m{36pt}}}
& & $x_1$ even & $x_1$ even & $x_1$ odd  & $x_1$ odd \\ 
& & $x_2$ even & $x_2$ odd  & $x_2$ even & $x_2$ odd 
\end{tabular}
\\[3ex]
\mu_2 : 
\left\{
\begin{tabular}{m{12pt}c*{4}{m{36pt}}@{}}
\unitcube{0} & $\mapsto$ & \fourcube{0}{1}{0}{0} & \fourcube{1}{2}{0}{2} 
& \fourcube{0}{3}{0}{2} & \fourcube{1}{0}{0}{0} \\ \\
\unitcube{1} & $\mapsto$ & \fourcube{0}{1}{1}{1} & \fourcube{1}{2}{1}{3} 
& \fourcube{0}{3}{1}{3} & \fourcube{1}{0}{1}{1} \\ \\
\unitcube{2} & $\mapsto$ & \fourcube{2}{1}{2}{0} & \fourcube{3}{2}{2}{2} 
& \fourcube{2}{3}{2}{2} & \fourcube{3}{0}{2}{0} \\ \\
\unitcube{3} & $\mapsto$ & \fourcube{2}{1}{3}{1} & \fourcube{3}{2}{3}{3} 
& \fourcube{2}{3}{3}{3} & \fourcube{3}{0}{3}{1} 
\end{tabular}
\right.
\end{array}
\label{s2mu}
\end{equation}
\normalsize As a seed to generate the 2-dimensional paperfolding
structure with the help of $\mu_2$, we may take
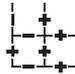
\begin{center}
\fourcube{3}{0}{1}{0}
\end{center}
where we put the origin at the centre. 

As in the 1-dimensional case, we can write the substitution $\mu_2$ symbolically,
as a block substitution on a 16 letter alphabet $B=\{ b_{ij}: 0\leq i,j\leq3 \}$.  
Here again, the first index indicates the crease pattern on the boundaries
of the semi-cube, and the second index the parities of the reference point,
in the order of the rows, resp.\  columns of the table in Eq.~(\ref{s2mu}).
Each letter is replaced by a block of letters as follows:
\small
\[
\nu_2 : \left\{
\begin{tabular}{@{}*{4}{r@{\,$\mapsto$}l}}
$b_{00}$ & \begin{tabular}{c@{\,}} $b_{01}$ $b_{13}$ \\ $b_{00}$ $b_{02}$ \end{tabular}, &
$b_{01}$ & \begin{tabular}{c@{\,}} $b_{11}$ $b_{23}$ \\ $b_{00}$ $b_{22}$ \end{tabular}, &
$b_{02}$ & \begin{tabular}{c@{\,}} $b_{01}$ $b_{33}$ \\ $b_{00}$ $b_{22}$ \end{tabular}, &
$b_{03}$ & \begin{tabular}{c@{\,}} $b_{11}$ $b_{03}$ \\ $b_{00}$ $b_{02}$ \end{tabular}, \\[3ex]
$b_{10}$ & \begin{tabular}{c@{\,}} $b_{01}$ $b_{13}$ \\ $b_{10}$ $b_{12}$ \end{tabular}, &
$b_{11}$ & \begin{tabular}{c@{\,}} $b_{11}$ $b_{23}$ \\ $b_{10}$ $b_{32}$ \end{tabular}, &
$b_{12}$ & \begin{tabular}{c@{\,}} $b_{01}$ $b_{33}$ \\ $b_{10}$ $b_{32}$ \end{tabular}, &
$b_{13}$ & \begin{tabular}{c@{\,}} $b_{11}$ $b_{03}$ \\ $b_{10}$ $b_{12}$ \end{tabular}, \\[3ex]
$b_{20}$ & \begin{tabular}{c@{\,}} $b_{21}$ $b_{13}$ \\ $b_{20}$ $b_{02}$ \end{tabular}, &
$b_{21}$ & \begin{tabular}{c@{\,}} $b_{31}$ $b_{23}$ \\ $b_{20}$ $b_{22}$ \end{tabular}, &
$b_{22}$ & \begin{tabular}{c@{\,}} $b_{21}$ $b_{33}$ \\ $b_{20}$ $b_{22}$ \end{tabular}, &
$b_{23}$ & \begin{tabular}{c@{\,}} $b_{31}$ $b_{03}$ \\ $b_{20}$ $b_{02}$ \end{tabular}, \\[3ex]
$b_{30}$ & \begin{tabular}{c@{\,}} $b_{21}$ $b_{13}$ \\ $b_{30}$ $b_{12}$ \end{tabular}, &
$b_{31}$ & \begin{tabular}{c@{\,}} $b_{31}$ $b_{23}$ \\ $b_{30}$ $b_{32}$ \end{tabular}, &
$b_{32}$ & \begin{tabular}{c@{\,}} $b_{21}$ $b_{33}$ \\ $b_{30}$ $b_{32}$ \end{tabular}, &
$b_{33}$ & \begin{tabular}{c@{\,}} $b_{31}$ $b_{03}$ \\ $b_{30}$ $b_{12}$ \end{tabular}, 
\end{tabular}
\right.
\label{s2nu}
\]
\normalsize
with seed \begin{tabular}{c@{\,}} $b_{32}$ $b_{00}$ \\ $b_{13}$ $b_{01}$ \end{tabular}.

\section{Properties of the Paperfolding Substitution}
\label{sec:properties}

In this Section we turn to the study of the properties of the
paperfolding substitution $\mu_d$. In particular, we must show
that it is {\em primitive}, which is defined as follows. 

\begin{definition}
  \label{def:primitive}
  Let $\rho$ be a symbolic block substitution on the alphabet
  $A=\{\alpha_1,\ldots,\alpha_m\}$. We say that $\rho$ is primitive if
  there is an integer $k$ such that for all $i$ the iteration
  $\rho^k(\alpha_i)$ contains all $\alpha_j$.
\end{definition}

In order to prove the primitivity, we must show that all types 
of semi-cube do occur in the generated structure at the same time. 
More precisely,
in each of the $2^d$ parity classes of sites, semi-cubes with all
$2^d$ combinations of creases on their faces must occur.

We begin by looking at sites $\boldsymbol{x}$, whose local crease 
pattern, on the faces containing $\boldsymbol{x}$, is a translated
and/or reflected copy of $S_d(1)$. This is the case for all sites
whose parities are all even or all odd. The signs of the creases
in $S_d(1)$, which is the result of the first $d$-fold, is given by
(\ref{eq:  prod def of c}). We shall now analyse the behaviour
of $S_d(1)$ under a sequence of reflections, that is, we look at
the signs of the creases of
\begin{equation}
    m_{a_1}'\ldots m_{a_r}' S_d(1). 
    \label{eq: seq of reflections of S1}
\end{equation}
Let $\boldsymbol{a}= a_1a_2\ldots a_r$ with $a_i \in {1,\ldots,d}$ be
the sequence of reflections of $S_d(1)$ from (\ref{eq: seq of
  reflections of S1}). Recall that the function $c$ is used to denote
the crease signs, and let us therefore generalise this notation to
$c_{\boldsymbol{a}}$, meaning the signs of the creases of $S_d(1)$
after a sequence of reflections, as in (\ref{eq: seq of reflections of S1}).

Let us start by considering the signs of the creases of $S_d(1)$ after
one reflection, that is, the signs of $m_j'S_d(1)$. As before let
$\boldsymbol{\sigma}^k$ label the creases. Recall (\ref{eq:  prod def of c})
that $c(\boldsymbol{\sigma}^k)$ changes the sign under 
$\sigma^k_j \mapsto -\sigma^k_j$ if $j<k$, and stays the same otherwise. 
As $m'_j$ reflects at $x_j=0$ and then reverts all creases, this gives
\[	
c_j(\boldsymbol{\sigma}^k) =
\begin{cases}
  \text{sign}  \Big( - \displaystyle{\prod_{i<k}} \sigma^k_{i}\Big) & 
     \textnormal{if } k \le j \vspace{2ex}\\
  \text{sign}  \Big( \displaystyle{ \prod_{i<k}} \sigma^k_{i}\Big) & 
     \textnormal{if } j < k 
\end{cases}
\]
For a sequence of reflections this generalises to 
\begin{equation}
	c_{\boldsymbol{a}}(\boldsymbol{\sigma}^k) =
 \text{sign}  \Big( (-1)^{N(\boldsymbol{a},k)} 
    \displaystyle{\prod_{i<k}} \sigma^k_{i}\Big) 
\label{eq: gen sign of walls} 
\end{equation}
with $N(\boldsymbol{a},k) = \big| \{t: k \leq a_t, 1\leq t\leq |\boldsymbol{a}|\}\big|$.
From (\ref{eq: gen sign of walls}) we see that we may choose the
elements of $\boldsymbol{a}$ to be unique and to be in strictly
increasing order, that is, $a_i<a_{i+1}$.

\vspace{2ex}

Let us now look at the signs of the creases in the first orthant 
(the orthant with $x_i>0$ for all $i$) of $S_d(1)$ and its reflections. 
The fold $S_d(1)$ divides the space into $2^d$ orthants, where the first
orthant contains the semi-cube with the origin as reference point.
The reflections of $S_d(1)$ will in general contain a different
semi-cube in their first orthant. Let $Q(d)$ be the set of all 
semi-cubes contained in the first orthant of some reflection of 
$S_d(1)$. Then we have the following result.

\begin{lemma}
  \label{lemma: all semi cubes}
  All $2^d$ types of semi-cube occur in the first orthant of some
  reflection of $S_d(1)$. In other words, $|Q(d)| = 2^d$.
\end{lemma}

\begin{proof}
As noted before, we have at most $2^d$ different semi-cubes, so that
$|Q(d)|\leq 2^d$. To prove equality, note first that the creases in
the semi-cube in the first orthant of $S_d(1)$ can be described by
$\boldsymbol{\sigma}^k \in \{0,1\}^d$ with precisely one 0.  Let
$\boldsymbol{a} = a_1a_2\ldots a_r$ and $\boldsymbol{b} = b_1b_2\ldots
b_t$ with $a_i,b_i\in \{1,\ldots,d\}$ be two sequences of reflections. 
We may, as noted above, assume that both sequences are strictly 
increasing. Assume that
\begin{equation}
c_{\boldsymbol{a}}(\boldsymbol{\sigma}^k) = c_{\boldsymbol{b}}(\boldsymbol{\sigma}^k)
\qquad
\text{for all }1\leq k\leq d,
\label{eq: assumption ca = cb}
\end{equation}
that is, two sequences of reflections give rise to the same semi-cube
in the first orthant. The strict monotonicity of the sequence 
$\boldsymbol{a}$ implies that
\[
0 \leq N(\boldsymbol{a},k-1) - N(\boldsymbol{a},k) \leq 1 
\quad \text{for } k = 2,\ldots, d, 
\]
and similar for the sequence $\boldsymbol{b}$. This at most one step
change on $N$ implies that $N(\boldsymbol{a},k) = N(\boldsymbol{b},k)$
for all $1\leq k\leq d$, since otherwise we would, 
via (\ref{eq: gen sign of walls}),
\[	(-1)^{N(\boldsymbol{a},k)} = (-1)^{N(\boldsymbol{b},k)},
\]
contradict the assumption (\ref{eq: assumption ca = cb}). Therefore
$\boldsymbol{a} = \boldsymbol{b}$, that is, two sequences of
reflections that give rise to the same semi-cube must be the same.
\end{proof}

Consider now the points $(\pm1,\pm1,\ldots,\pm1)$ in the crease
pattern $S_d(3)$. They have their parities all odd, and
they form an orbit of length $2^d$ under the reflection group.
Likewise, their local crease patterns also form an orbit under
the reflection group. As the local crease pattern at $(1,1,\ldots,1)$
is $S_d(1)$, by Lemma \ref{lemma: all semi cubes} we conclude that
this orbit has length $2^d$, and thus all $2^d$ semi-cube types
occur at this parity class of points. The same argument also
applies to the points $(\pm2,\pm2,\ldots,\pm2)$, with parities all 
even. 

Unfortunately, there are also points whose local crease pattern
is not a reflection of $S_d(1)$. However, the local crease pattern
of any point is obtained by a sequence of $d$ $1$-folds in $d$ 
different directions, in {\em some} order. These crease patterns 
are thus obtained from $S_d(1)$ (and its reflections) by a 
permutation of the main coordinate axes. The local crease patterns 
located on points in an orbit under the reflection group still
form an orbit of full length, so that also on these points, which
are all in the same parity class, all types of semi-cubes occur.
We therefore obtain the following Theorem:

\begin{theorem}
\label{thm: primitive}
The paperfolding substitution $\mu_d$ is primitive.
\end{theorem}

\begin{proof}
The crease pattern $S_d(3)$ contains points of all parity classes,
whose orbit under the reflection group has length $2^d$. The 
semi-cubes with origin in such an orbit of points (which are all
in the same parity class) occur with $2^d$ different crease patterns
on their faces. In other words, all semi-cube types occur with
reference points in all parity classes.
\end{proof}

We note that, independently of the dimension, three iterations
of the substitution are enough to generate all semi-cube types
from our standard seed, and four iterations from a single 
semi-cube of any type.

As indicated in the introduction, the existence of a generating
primitive substitution has a number of far reaching consequences 
for the $d$-dimensional paperfolding structures. These are all
standard for primitive substitution structures, compare \cite{BG}.

To start, one should first note that, instead of studying an 
individual structure, it has many advantages if one studies its 
{\em hull}. The hull $\Omega_d^{pf}$ of the $d$-dimensional paperfolding 
structures is given by the closure of the translation orbit of 
one $d$-dimensional paperfolding structure, with respect to a topology 
in which two patterns are $\epsilon$-close if they agree in a ball
$B_{1/\epsilon}$ centred at the origin, up to a translation of order
$\epsilon$. The hull is a compact space, and consists of all patterns 
locally indistinguishable from the starting one, i.e., all structures 
having the same finite subpatterns. Every element of the hull is repetitive,
i.e., given any finite subpattern, its translates occurring in the
structure form a relatively dense set (a set without holes of arbitrary
size). There is a natural translation action by homeomorphisms on 
the hull. That is, the hull, equipped with this group action, forms
a topological dynamical system. Due to the repetitivity of the elements
of the hull, this action is minimal (every orbit is dense), and it
is uniquely ergodic, admitting a unique translation invariant 
probability measure. This latter property is related to the fact
that any finite subpattern has a well-defined frequency.

An interesting question is whether we can say something on the
nature of the diffraction spectrum of $d$-dimensional paperfolding
structures, or, equivalently, on the dynamical spectrum of the
translation action on the hull (compare \cite{BG}). Having a
generating primitive block substitution is of help here, too.
For such structures, there is a simple criterion by Dekking 
\cite{dekking78}. If it is satisfied, the structure has pure point
dynamical and diffraction spectrum.

\begin{definition}
We say that the $d$-dimensional block substitution $\rho$ admits a 
coincidence in the sense of Dekking \cite{dekking78}, if there are 
integers $k$ and $t_1,\ldots,t_d$, such that for all $i$, the 
iteration $\rho^k(\alpha_i)$ has the same symbol $\alpha_j$ at 
position $(t_1,\ldots,t_d)$.
\end{definition}

\begin{theorem}
The $d$-dimensional paperfolding structures have pure point 
diffraction spectrum, and the associated dynamical system given
by the translation action on the hull has pure point dynamical
spectrum.
\end{theorem}
 
\begin{proof}
By \cite{dekking78}, it is enough to show that the paperfolding 
substitutions $\mu_d$ admit a coincidence in the sense of Dekking 
(for the higher-dimensional case, see also \cite{frank}).
Substituting a semi-cube with reference point at the origin once,
we obtain at position $(1,\ldots,1)$ a semi-cube with parities 
all odd. Substituting a second time, the semi-cube at 
$(1,\ldots,1)$ is mapped to a $2^d$ block at $(2,\ldots,2)$. 
The semi-cube at $(3,\ldots,3)$ in this block depends only on the 
parity of the semi-cube at $(1,\ldots,1)$ after the first substitution, 
which was all odd. So, the semi-cube at $(3,\ldots,3)$ is always the 
same, whatever the starting semi-cube was.
\end{proof}

Having a generating substitution also allows to estimate the
complexity of the generated structures. A natural complexity
measure is the growth of the number of distinct cubic subpatterns
of linear size $n$. The key observation here is, that there is
a hierarchy of semi-cubes of all orders in the structure.
The substitution maps a semi-cube to a block of $2^d$ semi-cubes,
which we call a super-semi-cube. This super-semi-cube is mapped
to a super-semi-cube of order two, and so on. The arrangement
of super-semi-cubes of any order now is, after rescaling, locally
indistinguishable from the arrangement of semi-cubes (the structure
is self-similar). In particular, the number of distinct $2^d$ blocks 
of super-semi-cubes of order $k$ is the same as the number of $2^d$ 
blocks of semi-cubes. This allows us to estimate the complexity function.

\begin{theorem}
\label{thm: growth}
  The number of distinct cubic subpatterns of linear size $n$ 
  in a $d$-dimensional paperfolding structure grows at most as
  $\text{const}\cdot n^d$.
\end{theorem}

\begin{proof}
Let $k$ be the smallest natural number such that $2^k\ge n$.
Any cubic subpattern of linear size $n$ is now contained in
some $2^d$ block of super-semi-cubes of order $k$. There is
a fixed number $c$ of such blocks (independent of $k$).
Let $\boldsymbol{x}$ be the corner of the $n^d$ block with
the lowest coordinates. To count all distinct $n^d$ blocks,
it is enough to let $\boldsymbol{x}$ vary within one 
super-semi-cube of order $k$. So, the number of distinct 
$n^d$ subpatterns is at most $c\cdot2^{kd}\le c\cdot(2n)^d$.
\end{proof}

Denote by $P_d(n)$ the number of cubic $n$-patterns in the $d$-dimensional
paperfolding structure, by which we mean semi-cubes with side length $n$ 
(with interior). By Theorem \ref{thm: growth} we know that
$P_d(n)=O(n^d)$. In \cite{allouche92} Allouche shows that $P_1(n) =
4n$ for $n\geq7$, and computer enumerations indicate that the
following holds for $n\geq3$:
\[ 
P_2(n) = 12n^2 -4 -16\cdot 2^{2\lfloor \log_2 (n-1) \rfloor} 
         + 24n\cdot2^{\lfloor \log_2 (n-1) \rfloor}.
\]

Finally, having a primitive substitution generating the 
paperfolding structures allows to determine topological
invariants of its hull. Anderson and Putnam \cite{AP} have 
shown that a generating primitive substitution allows to construct 
the hull as an inverse limit of a sequence of finite cell complexes 
and cellular maps between them. This in turn allows to compute the 
\v{C}ech cohomology as the direct limit of the cohomologies of these 
approximant cell complexes. We have a computer program that implements
this computation for arbitrary block substitutions in dimensions
$1$ and $2$. Applying it to the paperfolding substitutions, we obtain
the following results:

\begin{theorem}
  The hull of the classical 1-dimensional paperfolding structures has
  \v{C}ech cohomology groups
  \[ \check{H}^0=\Z,\quad\check{H}^1=\Z[\textstyle{\frac12}]\oplus\Z. \]
  The hull of the 2-dimensional generalised paperfolding structures has
  \v{C}ech cohomology groups
  \[ \check{H}^0=\Z,\quad\check{H}^1=\Z[\textstyle{\frac12}]^2,\quad
     \check{H}^2=\Z[\textstyle{\frac14}]\oplus
                 \Z[\textstyle{\frac12}]^2\oplus\Z^3\oplus\Z_2.
  \]
\end{theorem}

\section{Concluding Remarks}

In our paperfolding structures, we have always folded the half-space
with negative coordinates onto the half-space with positive coordinates,
first in direction 1, then 2, and so on. Ben-Abraham et al. \cite{benabraham}
have chosen the same convention. This could be varied, of course,
but the generated structures would be different, in general. 
For instance, one could change the order of the folding in 
the different directions, or fold underneath instead of ontop
of the positive half-space in some directions. One could even
vary the type of $d$-fold according to a periodic or aperiodic
sequence, along the lines of \cite{GaeMal13}. We suspect that as 
long as the sequence of folds applied is periodic, a generating 
substitution may still exist, but the situation with an 
aperiodic sequence of folds will be considerably more complicated. 

\section{Acknowledgement}
This work was supported by the German Research Council (DFG), via CRC 701.
We thank Alexey Garber helpful discussions.

\end{document}